\documentclass[reqno, 12pt]{amsart}

\usepackage{mathrsfs}
\usepackage{amscd}
\usepackage{amsmath}
\usepackage{latexsym}
\usepackage{amsfonts}
\usepackage{amssymb}
\usepackage{amsthm}
\usepackage{graphicx}
\usepackage{hyperref}
\usepackage{makecell}
\usepackage{array}
\usepackage{booktabs}
\usepackage{multirow}
\usepackage{color,xcolor}

\newtheorem{theorem}{Theorem}[section]
\newtheorem{assumption}[theorem]{Assumption}
\newtheorem{proposition}[theorem]{Proposition}

\newtheorem{lemma}[theorem]{Lemma}

\newtheorem{remark}[theorem]{Remark}

\numberwithin{equation}{section}

\title[Inverse random source problem]{Inverse random source problem for Maxwell equations in an inhomogeneous medium}

  {\author[T. Wang]{Tianjiao Wang}\address{School of Mathematical Sciences, Zhejiang University, Hangzhou 310058, China}}\email{wangtianjiao@zju.edu.cn}

\author[X. Xu]{Xiang Xu}
\address{School of Mathematical Sciences, Zhejiang University, Hangzhou 310058, China}
\email{xxu@zju.edu.cn}
\author[Y. Zhao]{Yue Zhao}
\address{School of Mathematics and Statistics, and Key Lab NAA-MOE, Central China Normal University,
Wuhan 430079, China}
\email{zhaoyueccnu@163.com}


\begin{document}

\begin{abstract}

This paper concerns the inverse random source problem of the stochastic Maxwell equations driven by white noise  in an inhomogeneous background medium.
The well-posedness is established for the direct source problem, and the estimates and regularity of the solution are obtained. 
A logarithmic stability estimate is established for the inverse problem of determining the strength of the random source. The analysis
only requires the random Dirichlet data at a fixed frequency.

\end{abstract}

\maketitle

\section{Introduction}

Inverse source problem in electromagnetic waves arise in many scientific and industrial areas such as antenna design and synthesis, biomedical imaging, and optical
tomography \cite{blz}.  It aims to determine the unknown source from appropriate measurements of the radiating fields.
For instance, the imaging modality for magnetoencephalography (MEG) in medical imaging is a non-invasive neurophysiological technique which measures the electric or magnetic fields generated by neuronal activity of the brain  \cite{abf, Nara}. The spatial distributions of the measured fields are analyzed to localize the source currents in the brain. Driven by the important applications, inverse source problems have attracted much attention by many researchers in both the engineering and the mathematical communities \cite{isakov}.

In most existing works, the sources have been considered to be deterministic functions for the inverse electromagnetic source scattering problems.
However, in situations involving unpredictable environments, incomplete knowledge of the system, and random measurement noise, it has been realized 
that random sources are able to better capture the uncertain behavior of the surrounding environments  \cite{De}. Recently,  uniqueness of an inverse random source problem was obtained in \cite{LW_Maxwell} for Maxwell equations in a homogeneous medium. 

In many practical applications, the background medium in which the electromagnetic waves propagate is not homogeneous. For instance, in the medical imaging of the brain, it is important to incorporate the sudden change of sound speed across the human head \cite{LZZ}. Moreover, it is possible to achieve some desirable radiation pattern that would otherwise not be realistically possible for a source embedded in free space \cite{MKB}. This phenomenon finds important applications in the antenna community in designing antenna embedding materials or substrates such as nonmagnetic dielectrics, magneto-dielectrics, and double negative meta-materials, in order to achieve specified electromagnetic radiation patterns. However, uniqueness and stability issues remain widely open for the inverse random source problems on Maxwell equations in inhomogeneous media.

We consider the mathematical study of the electromagnetic inverse random source problem in an inhomogeneous medium.
Consider the time-harmonic Maxwell equations
\begin{align}\label{me}
\nabla\times\boldsymbol E(x) - {\rm i}k\boldsymbol H(x) = 0, \quad
\nabla\times\boldsymbol H(x) + {\rm i}k n(x)\boldsymbol E(x) = \boldsymbol
J(x),\quad x\in\mathbb R^3,
\end{align}
where ${k}>0$ is the wavenumber, $n$ is the refractive index
and $m:=1-n \in C_0^\infty (B_R)$ with $B_R=\{x\in\mathbb R^3: |x|<R\}$.
$\boldsymbol J $ is the random source driven by an additive white noise of the form
\[
\boldsymbol J =\dot{\boldsymbol W}_x = \sqrt{\sigma(x)}(\dot W_1(x), \dot W_2(x), \dot W_3(x))^{\top},
\]
 where $\dot W_1(x), \dot W_2(x)$ and $\dot W_3(x)$ are three independent one-dimensional white noise
 which are defined on a complete probability space $(\Omega, {\mathcal A}, \mathbb P)$, and $\sigma(x)\in C_0^\infty(B_R)$ is the strength of the random source.
Eliminating the magnetic field $\boldsymbol
H$ from \eqref{me}, we obtain the decoupled Maxwell system for the electric
field $\boldsymbol E$:
\begin{equation}\label{ef}
 \nabla\times(\nabla\times\boldsymbol E)-{k}^2 n \boldsymbol E={\rm i}{k}\boldsymbol J\quad\text{in}~ \mathbb R^3.
\end{equation}
In addition, the following Silver--M\"{u}ller radiation condition is required
at infinity to ensure the uniqueness of the direct problem: 
\begin{equation}\label{src}
\lim_{r\to \infty}((\nabla\times\boldsymbol E) \times
 \widehat x - {\rm i}{k} r \boldsymbol E) = 0, \quad r=|x|.
\end{equation}
Denote the boundary of $B_R$ by $\partial B_R$.
The inverse problem is to determine the strength $\sigma$ of the random source from the tangential trace of the random electric field 
$\boldsymbol E\times \boldsymbol \nu$
measured on $\partial B_R$ with $\boldsymbol \nu$ being the outward unit normal vector on $\partial B_R$.

 Compared with the deterministic counterparts, if the random source is a rough field which does not exist pointwisely, then the random source should be understood
 in the sense of distribution. As a result, the stochastic Maxwell equation should also be studied in distributional sense. 
 To deal with this issue, the random source was required to satisfy a divergence free condition in \cite{LW_Maxwell}, and then the Maxwell equation \eqref{ef}
 reduced to a vector-valued Helmholtz equation whose well-posedness and regularity have been studied in \cite{LLW}. However, in our setting the random source is an additive white noise which does not satisfy the divergence free condition. Hence, new method has to be developed to establish the well-posedness of 
 the direct scattering problem. For the inverse random source problems  driven by white noise, uniqueness and stability were achieved for acoustic waves in
 \cite{LL}. But the analysis there cannot be carried over to Maxwell equations, since the Green tensor of Maxwell equations is more singular than that
 of the Helmholtz equation. Due to the above difficulties, the inverse random source problem for Maxwell equations driven by an additive white noise is completely open so far.
 This work initiates the mathematical analysis of the direct and inverse source scattering problems for the stochastic Maxwell equations driven by an additive white noise.
 
 In this paper, we establish the well-posedness for \eqref{ef} in the sense of distribution and obtain the regularity for the electric field. The key ingredient in the 
 analysis is employing the scattering theory to investigate the resolvent of the differential operator defined on Sobolev space with negative smooth index.
 With the help of the resolvent estimates and the transparent boundary condition, we establish an integral equation which connects the random source and the tangential
 trace of the electric field on the boundary $\partial B_R$. Based on the integral equation, for the first time, the stability is derived for the inverse source problem of Maxwell equations driven by an additive white noise in inhomogeneous media. The proof employs Ito isometry
 and constructions of complex geometric optics solutions to the associated Maxwell equation. 
 The analysis requires the Dirichlet data only at a fixed frequency. Our methods are unified which can be applied to other stochastic wave equations
 in inhomogeneous media or with potential functions.
 
 Stability for the inverse random source problem obtained in this work exhibits some features which differ significantly from its deterministic counterparts.
As a matter of fact, general uniqueness does not hold for the deterministic inverse source problems of Maxwell equations even if multi-frequency data is given, unless the vector-valued source is divergence free \cite{Monk, Zhao}. 
 This phenomenon may be attributed to the fact that for stochastic inverse source problems, it is less meaningful to characterize the source by a particular
 realization. It is the statistics such as the mean and variance that are of more interests 
 in stochastic inverse problems, which are used to quantify the uncertainties of the random sources.
  
 This paper is organized as follows. In Section \ref{dp1}, we study the well-posedness of the direct source problem. Section \ref{ip}
 is devoted to the inverse random source problem. An integral equation is established which connects the random source and the electric field on the boundary.  Based on the integral equation, a logarithmic stability estimate is derived for the inverse random source problem.

\section{Direct scattering problem}\label{dp1}

In this section, we establish the well-posedness of the direct scattering problem \eqref{ef}--\eqref{src} and give the regularity of the solution.

We introduce some function spaces.
Denote the standard Sobolev spaces by $H^s(\mathbb R^3):=W^{s,2}(\mathbb R^3)$. The function spaces $H^s_{ loc}:=H^s_{  loc}(\mathbb{R}^3)$ and $H^{s}_{comp}:=H^s_{ comp}(\mathbb{R}^3)$ are respectively defined by \begin{align*}
	H^s_{loc}(\mathbb{R}^3) &:=\{u :\chi u \in H^s(\mathbb{R}^3),\forall \, \chi \in C_0^\infty(\mathbb{R}^3)\}, \\ H^s_{  comp}(\mathbb{R}^3) &:=\{u\in H^s(\mathbb R^3) :\exists \, \chi \in C_0^\infty(\mathbb{R}^3),\,\chi u =u\}.
\end{align*} 
Denote $H(curl):=\{\boldsymbol u: \boldsymbol u\in L^2(\mathbb R^3)^3, \nabla\times \boldsymbol u\in L^2(\mathbb R^3)^3\}$ and then
\begin{align*}
H_{loc}(curl):&=\{\boldsymbol u: \chi \boldsymbol u \in H(curl),\forall \, \chi \in C_0^\infty(\mathbb{R}^3)\}.
\end{align*}

The bounded linear operators between two function spaces $X$ and $Y$ are denoted by $\mathcal{L}(X,Y)$ with the operator norm $\|\cdot\|_{\mathcal{L}(X,Y)}$.  Moreover, we use $a \lesssim b $ to stand for $a \le C b$, where $C>0$ is a generic constant whose specific value is not required but should be clear from the context. 

We introduce the operator $\mathcal E_0(t)$ which solves the equation
\[
\square_0 \mathcal E_0(t)=0,\quad \mathcal E_0(0)=0,\quad \partial_t\mathcal E_0(0)={\mathbb I},
\]
where $\square_0:=\partial^2_t+\nabla\times\nabla\times\cdot$ and $\mathbb I$ is the identity operator. As a consequence, 
given a source $\boldsymbol f\in L^2_{comp}(\mathbb R^3)^3$,
$\boldsymbol U_0(x, t) = \mathcal E_0(t)(\boldsymbol f)$ satisfies the following decoupled Maxwell equation
\[
\partial_{t}^2U_0(x, t) +\nabla\times(\nabla\times  \boldsymbol U_0(x, t)) = 0
\]
with
\[
\boldsymbol U_0(x, 0) = 0, \quad \partial_t \boldsymbol U_0(x, 0) = \boldsymbol f.
\]
It follows from  \cite[Section 7.2]{acl}
\begin{align}\label{regularity}
\mathcal E_0(t)(\boldsymbol f)\in L^2((0, T), H_{loc}(curl)), \quad \partial_t\mathcal E_0(t)(\boldsymbol f)\in L^2((0, T), L^2_{loc}(\mathbb R^3)^3).
\end{align}
Denoting the free resolvent by $R_0(\lambda) = (\nabla\times(\nabla\times\cdot)- \lambda^2)^{-1}$,
we have the representation of $R_0(\lambda)$ as follows
 \[
 R_0(\lambda)=\int_0^\infty e^{ {\rm i}\lambda t} \mathcal E_0(t)\,\mathrm{d}t.
 \] 
Let $\chi\in C^\infty_0(\mathbb R^3)$ be a cutoff function. The strong Huygens's principle implies that 
\[
 (\mathcal E_0(t) \chi)(x)=0,\quad \text{when} \,\, t>\sup\{|x-y|:y \in \text{supp} \chi\}.
 \]  Then we obtain \[
\chi  R_0(\lambda) \chi =\int_0^L e^{{\rm i}\lambda t} \chi \mathcal E_0(t) \chi \,\mathrm{d}t,
\] where $L> \text{diam}\,\text{supp} \chi$.
  Furthermore, we have 
 \begin{align}\label{fre}
\chi  R_0(\lambda) \chi=({\rm i}\lambda)^{-1}\int_0^L  \partial_t(e^{{\rm i}\lambda t}) \chi \mathcal E_0(t) \chi \,\mathrm{d}t=-({\rm i}\lambda)^{-1}\int_0^L  e^{{\rm i}\lambda t} \chi  \partial_t\mathcal E_0(t) \chi \,\mathrm{d}t.
\end{align}
Combing \eqref{regularity} and \eqref{fre} we obtain that 
\begin{align*}
\chi  R_0(\lambda)\chi : L^2(\mathbb R^3)^3  \to L^2(\mathbb R^3)^3
\end{align*}
is an analytic family of operators for $\lambda \in \mathbb C$ with the resolvent estimate
\[
\|\chi R_0(\lambda)\chi\|_{\mathcal{L}(L^2(\mathbb R^3)^3, L^2(\mathbb R^3)^3)}\lesssim\frac{1}{|\lambda|}.
\]

Let $\mathbf G (x, y)$ be the Green
tensor for the Maxwell system. Explicitly, we have
\begin{equation*}
 \mathbf G (x, y)= {\rm i} \lambda
g(x, y)\mathbf I_3+\frac{\rm i}{\lambda}\nabla_{x}\nabla_{x}^\top g(x,
y),
\end{equation*}
where $g(x, y)=\frac{1}{4\pi}\frac{e^{{\rm i}\lambda|x-y|}}{|x-y|}$ is the fundamental solution of the three-dimensional Helmholtz
equation and $\mathbf I_3$ is the $3\times 3$ identity matrix. Therefore, the resolvent $R_0(\lambda): L^2_{comp}(\mathbb{R}^3)^3\to H_{loc}(curl)$ can be represented as follows 
\begin{align}\label{ie}
R_0(\lambda)( \boldsymbol f) = \mathbf G * \boldsymbol f =\int_{\mathbb R^3}\mathbf G (x, y) \boldsymbol f(y){\rm d}y, \quad \boldsymbol f \in L_{comp}^2(\mathbb{R}^3)^3, \,\,\lambda
\in\mathbb C.
\end{align}

Now we intend to extend the domain of the resolvent $\chi R_0(\lambda)\chi$ from $L^2$ to $H^s$ and derive the corresponding resolvent estimates.
If $\boldsymbol f \in H^1(\mathbb{R}^3)^3$ and $|\alpha|=1$, we have
	\begin{align}
		\partial^\alpha (\chi R_0(\lambda) \chi \boldsymbol f) &=\partial^\alpha(\chi ( \mathbf G * (\chi \boldsymbol f)))=(\partial^\alpha  \chi)( \mathbf G * (\chi \boldsymbol f))+\chi ( \mathbf G * \partial^\alpha (\chi \boldsymbol f)) \notag\\ &=(\partial^\alpha  \chi)( \mathbf G * (\chi \boldsymbol f))+ \chi ( \mathbf G * \partial^\alpha \chi \boldsymbol f)+\chi ( \mathbf G * \partial^\alpha \boldsymbol f \chi ), \label{Cut1}
	\end{align}  
	which implies $\chi R_0(\lambda) \chi \boldsymbol f\in H^1(\mathbb R^3)^3$. Then by induction we have $\chi R_0(\lambda) \chi \boldsymbol f\in H^n(\mathbb R^3)^3$
	for $\boldsymbol f\in H^n(\mathbb R^3)^3,$ where $n$ is any positive integer.
	Further, using the interpolation of fractional order Sobolev spaces, i.e., \[H^s(\Omega)=[L^2(\Omega),H^n(\Omega)]_{s/n}, \quad  n>s>0,\quad n \in \mathbb{N},\] we have that 
	$\chi R_0(\lambda)\chi: H^s(\mathbb R^3)^3\to H^s(\mathbb R^3)^3$ for all $s\geq 0$ with resolvent estimates
	\[
\|\chi R_0(\lambda) \chi\|_{\mathcal L(H^s(\mathbb R^3)^3, H^s(\mathbb R^3)^3)}\lesssim \frac{1}{|\lambda|},\quad s\geq 0.
\]

Now we consider the case $s<0$. In this case, the integral representation \eqref{ie} should be treated as  a distribution in the following sense but not a pointwisely defined function:
\[
\langle  \mathbf G * \boldsymbol f, \boldsymbol \phi\rangle = \langle  \boldsymbol f, \mathbf G * \boldsymbol \phi\rangle, \quad \boldsymbol \phi \in C_0^\infty(\mathbb R^3)^3,
\]
where $\langle \cdot, \cdot\rangle$ denotes the inner product in the Hilbert space $L^2(\mathbb R^3)^3$.
Using the definition of convolution of distributions and the symmetry $\mathbf G(x, y) = \mathbf G(y, x)$ of the Green tensor, we obtain \begin{align*}
		\left\langle\chi R_0(\lambda) \chi \boldsymbol f, \boldsymbol \phi\right\rangle &=\left\langle \mathbf G*  (\chi  \boldsymbol f),\chi \boldsymbol\phi  \right\rangle=\left\langle \chi  \boldsymbol f, \left\langle\mathbf G(x-y), \boldsymbol\phi \chi \right\rangle_x \right\rangle_y \\ &=\left\langle  \chi  \boldsymbol f , \left\langle\mathbf G(y-x),\boldsymbol\phi \chi  \right\rangle_x \right\rangle_y
		\\&= \left\langle \chi \boldsymbol f , \mathbf G * (\boldsymbol\phi \chi) \right\rangle=\left\langle  \boldsymbol f,\chi R_0(\lambda)\chi \boldsymbol\phi \right\rangle, \quad \boldsymbol\phi \in C_0^\infty(\mathbb{R}^3)^3.
	\end{align*}
	Then letting $\boldsymbol\varphi\in H^{-s}(\mathbb R^3)^3$ and approximating it by a sequence of smooth functions $\{\boldsymbol\varphi_j\}_{j=1}^\infty$ we obtain
	\begin{align}\label{3}
	\left\langle\chi R_0(\lambda) \chi \boldsymbol f, \boldsymbol \varphi_j\right\rangle = \left\langle  \boldsymbol f,\chi R_0(\lambda)\chi \boldsymbol\varphi_j \right\rangle.
	\end{align}
	As $\chi R_0(\lambda)\chi \boldsymbol\varphi_j\in H^{-s}(\mathbb R^3)^3$ with 
	$\|\chi R_0(\lambda)\chi \boldsymbol\varphi_j\|_{H^{-s}(\mathbb R^3)^3}\lesssim \frac{1}{|\lambda|}\|\boldsymbol\varphi_j\|_{H^{-s}(\mathbb R^3)^3}$,
	we have
	\[
	|\left\langle  \boldsymbol f,\chi R_0(\lambda)\chi \boldsymbol\varphi_j \right\rangle|\lesssim \frac{1}{|\lambda|}
	\| \boldsymbol f\|_{H^{s}(\mathbb R^3)^3}\|\boldsymbol\varphi_j\|_{H^{-s}(\mathbb R^3)^3}
	\] 
	which gives by letting $j\to\infty$
	that the resolvent $\chi R_0(\lambda) \chi: H^s(\mathbb R^3)^3\to H^s(\mathbb R^3)^3$ for $s< 0$ is a bounded operator with the corresponding resolvent estimate 
\[
\|\chi R_0(\lambda) \chi\|_{\mathcal L(H^s(\mathbb R^3)^3, H^s(\mathbb R^3)^3)}\lesssim \frac{1}{|\lambda|},\quad s<0.
\]

		In summary, we obtain the following theorem, which concerns the analyticity and estimates of the free resolvent.
\begin{proposition}\label{free}
Let $\chi\in C_0^\infty(\mathbb R^3)$. The free resolvent $\chi R_0(\lambda) \chi: H^s\to H^s$ is an analytic family of bounded operators for $\lambda\in\mathbb C$
with the resolvent estimates
\[
\|\chi R_0(\lambda) \chi\|_{\mathcal L(H^s, H^s)}\lesssim \frac{1}{|\lambda|}, \,\,s\in\mathbb R.
\]
\end{proposition}

In the following we consider the resolvent of the Maxwell equation \eqref{ef} in an inhomogeneous medium
\[
R(\lambda) = (c^2(x)\nabla\times(\nabla\times\cdot)-\lambda^2)^{-1}.
\]
Denote $\frac{1}{n(x)} = c^2(x)$. We assume that the medium satisfies the following non-trapping condition.
\begin{assumption}\label{ntrap}
Let $\mathcal E(t)$ solve
the wave equation
\[
\square \mathcal E(t)=0,\quad \mathcal E(0)=0,\quad \partial_t \mathcal E(0)=\mathbb I,
\]
where $\square:=\partial^2_t+c^2(x)\nabla\times(\nabla\times\cdot)$.
For any $a>R$, $\exists \, T_a>0$ such that for each $\chi\in C_0^\infty
(B_a), \chi|_{B_{R+\tau}}=1$ with $\tau$  being a small constant, it holds that 
\[
\chi \mathcal E(t)\chi|_{t>T_a} \in C^\infty ((T_a, \infty);\mathcal
L(L^2(\mathbb R^3)^3, L^2(\mathbb R^3)^3)).
\]
\end{assumption}

The following proposition concerns the analyticity and estimate of the resolvent $R(\lambda)$. The proof adapts the arguments in \cite[Theorem 4.43]{DZ}.
\begin{proposition}\label{prop_resolvent}
For any $M>0$, there exists $C_0>0$ such that $R(\lambda)$ is holomorphic
in the domain
\[
\Omega_M=\{\lambda\in\mathbb{C}:\,\mathrm{Im}\lambda\geq -M\log|\lambda|,\quad |\lambda|\geq C_0\}.
\]
Assume $\chi\in C_c^\infty(\mathbb{R}^3)$ such that $\chi=1$ near $B_R$, the
following resolvent estimate holds:
\begin{equation}\label{stability_resolvent}
\|\chi R(\lambda)\chi\|_{L^2(\mathbb R^3)^3\to L^2(\mathbb R^3)^3}\leq
\frac{C}{|\lambda|}e^{T({\rm Im}\lambda)_-},
\end{equation}
for $\lambda\in \Omega_M$, where $C, T$ are positive constants and $({\rm
Im}\lambda)_-:=\max(0,-{\rm Im}\lambda)$. 
\end{proposition}

Let $\mathbf G_n(x, y)$ be the Green tensor of the Maxwell equations in inhomogeneous medium which satisfies
\[
\nabla\times(\nabla\times\mathbf G_n)-{k}^2 n \mathbf G_n=\mathbf I_3 \delta(x-y).
\]
It can be decomposed as follows
\[
\mathbf G_n(x, y) = \mathbf G(x, y) + \mathbf P(x, y).
\]
Here $\mathbf P(x, y)$ is a smooth matrix with $\mathbf P(x, y) = \mathbf P(y, x). $
Using Proposition \ref{prop_resolvent} and repeating the arguments for the proof of Proposition \ref{free}, we obtain the following theorem
for the resolvent of the Maxwell equation \eqref{ef} in an inhomogeneous medium.

\begin{theorem}\label{dp}
For any $M>0$, there exists $C_0>0$ such that $R(\lambda)$ is holomorphic
in the domain
\[
\Omega_M=\{\lambda\in\mathbb{C}:\,\mathrm{Im}\lambda\geq -M\log|\lambda|,\quad |\lambda|\geq C_0\}.
\]
Assume $\chi\in C_c^\infty(\mathbb{R}^3)$ such that $\chi=1$ near $B_R$. The
following resolvent estimates hold:
\begin{equation}\label{stability_resolvent}
\|\chi R(\lambda)\chi\|_{H^s(\mathbb R^3)^3\to H^s(\mathbb R^3)^3}\leq
\frac{C}{|\lambda|}e^{T({\rm Im}\lambda)_-},
\end{equation}
for $\lambda\in \Omega_M$, where $C, T$ are constants and $({\rm
Im}\lambda)_-:=\max(0,-{\rm Im}\lambda)$. 
\end{theorem}

As a corollary  of Theorem \ref{dp}, we obtain the well-posedness of the direct problem.
\begin{theorem}
Assume the medium is non-trapping. The direct scattering problem \eqref{ef}--\eqref{src} admits a unique solution $\boldsymbol E\in H_{loc}^{-3/2-\delta}(\mathbb R^3)^3$ almost surely
with the estimate
\[
\|\boldsymbol E\|_{H_{loc}^{-3/2-\delta}(\mathbb R^3)^3}\lesssim \|\boldsymbol J\|_{H^{-3/2-\delta}(\mathbb R^3)^3}.
\]
\end{theorem}
\begin{proof}
It is known that $\boldsymbol J\in H^{-3/2-\delta}(\mathbb R^3)^3$ where $\delta$ is any positive constant by the regularity of the white noise.
Then the existence and estimate of the solution is a direct consequence of Theorem \ref{dp}. The uniqueness is obtained by utilizing the uniqueness result
established in \cite{CK} for the deterministic counterpart.  
\end{proof}

\section{Inverse random source problem}\label{ip}

In this section, we study the inverse random source problem. The goal is to derive a stability estimate of determining $\sigma$ from the tangential trace of the electric
field on $\partial B_R$.

Let $\boldsymbol E\times\boldsymbol\nu$ and $\boldsymbol H\times\boldsymbol\nu$
be the tangential trace of the electric field and the magnetic field,
respectively. It is shown in \cite{CK}  that there exists a capacity
operator $T_{\rm M}$ such that
\begin{equation}\label{mtbc1}
 \boldsymbol H\times\boldsymbol\nu=T_{\rm M}(\boldsymbol
E\times\boldsymbol\nu)\quad\text{on}~\partial B_R,
\end{equation}
which implies that $\boldsymbol H\times\boldsymbol\nu$ can be computed once
$\boldsymbol E\times\boldsymbol\nu$ is available on $\partial B_R$. The transparent
boundary condition \eqref{mtbc1} can be equivalently written as
\begin{equation}\label{mtbc2}
(\nabla\times\boldsymbol E)\times\boldsymbol\nu={\rm i}k T_{\rm M}
(\boldsymbol E\times\boldsymbol\nu)\quad\text{on}~\partial B_R.
\end{equation}

In what follows,
we establish an integral equation connecting $\sigma$ and the Dirichlet boundary data.
Denote $\boldsymbol f = {\rm i}k \boldsymbol J$ and let $\boldsymbol f_\varepsilon$ be the mollifier of $\boldsymbol f$. Since $\boldsymbol f\in H^{-3/2-\delta}(\mathbb R^3)^3$ with compact support, $\boldsymbol f_\varepsilon$ is smooth and $\boldsymbol f_\varepsilon\to \boldsymbol f$ in $H^{-3/2-\delta}(\mathbb R^3)^3$.
For each $\boldsymbol f_\varepsilon$, there exists a unique radiating solution $\boldsymbol E_\varepsilon$ to \eqref{ef} such that
\begin{equation}\label{eqn_1}
\nabla\times(\nabla\times\boldsymbol E_\varepsilon)-{k}^2 n \boldsymbol E_\varepsilon=\boldsymbol f_\varepsilon\quad\text{in}~ \mathbb R^3.
\end{equation}

Let $\boldsymbol U$ satisfy the homogeneous Maxwell equation
\[
 \nabla\times(\nabla\times\boldsymbol U)-{k}^2 n \boldsymbol U=0\quad\text{in}~ \mathbb R^3.
\]
Multiplying both sides of the equation by $\boldsymbol U$ and integrating by parts yields
\[
\int_{B_R}\boldsymbol f_\varepsilon \cdot \boldsymbol U dx = 
-\int_{\partial B_R} {\rm i}kT_{\rm M} (\boldsymbol E_\varepsilon\times\boldsymbol\nu)\cdot \boldsymbol U + (\boldsymbol E_\varepsilon\times\boldsymbol\nu)\cdot (\nabla\times\boldsymbol U)  {\rm d}s(x).
\]
Then by letting $\varepsilon\to 0$ we obtain
\begin{align}\label{1}
\lim_{\varepsilon\to 0}\int_{B_R}\boldsymbol f_\varepsilon \cdot\boldsymbol U dx = \int_{B_R}\boldsymbol f \cdot\boldsymbol U dx.
\end{align}
We also have
\begin{align}\label{2}
&\lim_{\varepsilon\to 0}
\int_{\partial B_R} {\rm i}kT_{\rm M} (\boldsymbol E_\varepsilon\times\boldsymbol\nu)\cdot \boldsymbol U + (\boldsymbol E_\varepsilon\times\boldsymbol\nu)\cdot (\nabla\times\boldsymbol U)  {\rm d}s(x)\notag\\
&= \int_{\partial B_R} {\rm i}kT_{\rm M} (\boldsymbol E\times\boldsymbol\nu)\cdot \boldsymbol U + (\boldsymbol E\times\boldsymbol\nu)\cdot (\nabla\times\boldsymbol U)  {\rm d}s(x).
\end{align}
To prove the above limit, note that
\begin{align*}
&\Big|\int_{\partial B_R} T_{\rm M} (\boldsymbol E_\varepsilon\times\boldsymbol\nu)\cdot \boldsymbol U - T_{\rm M} (\boldsymbol E\times\boldsymbol\nu)\cdot \boldsymbol U {\rm d}s(x)\Big|\\
&\leq \|T_{\rm M} (\boldsymbol E_\varepsilon\times\boldsymbol\nu) - T_{\rm M} (\boldsymbol E\times\boldsymbol\nu)\|_{L^2(\partial B_R)^3}\|\boldsymbol U\|_{L^2(\partial B_R)^3}\\
&\leq 
\|\boldsymbol E_\varepsilon -  \boldsymbol E\|_{H^{3/2}(\Omega)^3}\|\boldsymbol U\|_{H^{3/2}(\partial B_R)^3}.
\end{align*}
Here $\Omega$ is a bounded domain satisfying $\partial B_R\subset \Omega$ and $\Omega\cap supp\boldsymbol J=\emptyset$.
Because near $\partial B_R$ the Maxwell equation reduces to the Helmholtz equation, which implies
\[
-\Delta (\boldsymbol E_\varepsilon - \boldsymbol E) - k^2 (\boldsymbol E_\varepsilon - \boldsymbol E) =0,
\]
from the resolvent estimate in Corollary \ref{dp} by letting $s=-3/2-\delta$ and interior elliptic regularity theory \cite[(7.13)]{Stein}, we have for any $m\in\mathbb N$
\begin{align*}
&\|\boldsymbol E_\varepsilon -  \boldsymbol E\|_{H^{-3/2-\delta+2m}(\Omega)^3}\\
&\leq C(m) \|\boldsymbol E_\varepsilon-\boldsymbol E\|_{H^{-3/2-\delta}(\Omega)^3}\\
&\leq C(m) \|\boldsymbol f_\varepsilon-\boldsymbol f\|_{H^{-3/2-\delta}(B_R)^3} \to 0 \,\,\text{as} \,\,\varepsilon\to 0,
\end{align*}
which gives the following limit by letting $m=2$
\[
\lim_{\varepsilon\to 0}\int_{\partial B_R} T_M (\boldsymbol E_\varepsilon\times\boldsymbol\nu)\cdot \boldsymbol U {\rm d}s(x) = \int_{\partial B_R} T_M (\boldsymbol E\times\boldsymbol\nu)\cdot\boldsymbol U {\rm d}s(x).
\]
In a similar way we have
\[
\lim_{\varepsilon\to 0}\int_{\partial B_R} (\boldsymbol E_\varepsilon\times\boldsymbol\nu)\cdot (\nabla\times\boldsymbol U)  {\rm d}s(x) = \int_{\partial B_R} (\boldsymbol E\times\boldsymbol\nu)\cdot (\nabla\times\boldsymbol U)  {\rm d}s(x),
\]
which yields \eqref{2}.
Combining \eqref{1} and \eqref{2} we establish the following integral equation
\begin{align}\label{ibp}
\int_{B_R}\boldsymbol f \cdot\boldsymbol U dx = \int_{\partial B_R} {\rm i}kT_{\rm M} (\boldsymbol E\times\boldsymbol\nu)\cdot \boldsymbol U + (\boldsymbol E\times\boldsymbol\nu)\cdot (\nabla\times\boldsymbol U)  {\rm d}s(x). 
\end{align}

Based on the integral identity \eqref{ibp}, we obtain the following lemma which plays a crucial role in the stability estimate. It establishes an estimate of the strength of the random source by the measurement of the electric field.
\begin{lemma}
For two smooth solutions $\boldsymbol U_1$ and $\boldsymbol U_2$ to the Maxwell equations \begin{align} \label{Maxw}
    \nabla \times (\nabla \times \boldsymbol U_j)-k^2n \boldsymbol U_j=0,\quad j=1,2,\end{align} 
there exists a positive constant $C$ dependent on $k$ and $R$ such that
\begin{align}
   \left| \int_{\mathbb R^3}\sigma \boldsymbol U_1^{\top} \boldsymbol U_2{\rm d}x \right| 
   \leq C\epsilon \|\boldsymbol U_1\|_{L^2( B_{R'})^3}\|\boldsymbol U_2\|_{L^2(B_{R'})^3}
, \label{key}
\end{align} where $R'>R$ and the boundary data $\epsilon$ is given by \[
\epsilon= \max\{\|\mathbb F_1\|_{\partial B_R \times \partial B_R},\|\mathbb F_2\|_{\partial B_R \times \partial B_R},\|\mathbb F_3\|_{\partial B_R \times \partial B_R}\}
\] with the norm $\|\cdot\|_{\partial B_R \times \partial B_R}$ for matrix-valued functions defined as \[
\|\mathbb A\|_{\partial B_R \times \partial B_R}:= \sum_{i,j=1}^3 \|A_{ij}\|_{L^2(\partial B_R \times \partial B_R)}.
\] 
Here
\begin{align*}
    \mathbb F_1(x,y) &:=\mathbb E[(\boldsymbol{E}(x) \times \boldsymbol\nu(x))( \boldsymbol{E}(y) \times \boldsymbol\nu(y))^\top],\\ 
    \mathbb F_2(x,y) &:= \mathbb E[ T_{\rm M}
(\boldsymbol E(x)\times\boldsymbol\nu(x))( \boldsymbol{E}(y) \times \boldsymbol\nu(y))^\top], \\
    \mathbb F_3(x,y) &:=\mathbb E[ T_{\rm M}
(\boldsymbol E(x)\times\boldsymbol\nu(x)) T_{\rm M}
(\boldsymbol E(x)\times\boldsymbol\nu(x))^\top].
\end{align*} 
\end{lemma}

\begin{proof}

Using the Ito isometry we obtain
\begin{align}\label{Iso}
\mathbb E\Big[\int_{\mathbb R^3}\boldsymbol J\cdot \boldsymbol U_1{\rm d}x\int_{\mathbb R^3}\boldsymbol J\cdot \boldsymbol U_2{\rm d}x\Big]
=\int_{\mathbb R^3} \sigma\boldsymbol U_1^{\top}\boldsymbol U_2{\rm d}x.
\end{align}
On the other hand, by \eqref{ibp} we have
\begin{align}\label{IBP}
{\rm i}k\int_{B_R}\boldsymbol J\cdot\boldsymbol U_j dx = \int_{\partial B_R} {\rm i}kT_{\rm M} (\boldsymbol E\times\boldsymbol\nu)\cdot \boldsymbol U_j + (\boldsymbol E\times\boldsymbol\nu)\cdot (\nabla\times\boldsymbol U_j)  {\rm d}s(x), \quad j=1, 2. 
\end{align}
Then combining \eqref{Iso} and \eqref{IBP} we obtain \begin{align*}
    &\int_{\mathbb R^3} \sigma\boldsymbol U_1^{\top}\boldsymbol U_2{\rm d}x\\
    &=\mathbb E \Big{[}\int_{\partial B_R} (\nabla \times \boldsymbol U_1) \cdot (\boldsymbol{E} \times \boldsymbol\nu)+T_{\rm M} (\boldsymbol E\times\boldsymbol\nu) \cdot \boldsymbol U_1\,\mathrm{d}s(x) \\ 
    &\quad\times\int_{\partial B_R} (\nabla \times \boldsymbol U_2) \cdot (\boldsymbol{E} \times \boldsymbol\nu)+T_{\rm M} (\boldsymbol E\times\boldsymbol\nu) \cdot \boldsymbol U_2\,\mathrm{d}s(y)\Big{]} \\ 
    &=\int_{\partial B_R}\int_{\partial B_R} \mathbb E[(\boldsymbol{E}(x) \times \boldsymbol\nu(x))( \boldsymbol{E}(y) \times \boldsymbol\nu(y))^\top] \otimes((\nabla \times \boldsymbol U_1(x))(\nabla \times \boldsymbol U_2(y))^{\top})\,\mathrm{d}s(x)\,\mathrm{d}s(y)\\ 
    &\quad+\int_{\partial B_R}\int_{\partial B_R} \mathbb E[T_{\rm M} (\boldsymbol E(x)\times\boldsymbol\nu(x))( \boldsymbol{E}(y) \times \boldsymbol\nu(y))^\top] \otimes(\boldsymbol U_1(x)(\nabla \times \boldsymbol U_2(y))^{\top})\,\mathrm{d}s(x)\,\mathrm{d}s(y)\\ 
    &\quad+\int_{\partial B_R}\int_{\partial B_R} \mathbb E[(\boldsymbol{E}(x) \times \boldsymbol\nu(x))T_{\rm M} (\boldsymbol E(y)\times\boldsymbol\nu(y))^\top] \otimes((\nabla \times \boldsymbol U_1(x))( \boldsymbol U_2(y))^{\top})\,\mathrm{d}s(x)\,\mathrm{d}s(y)\\ 
    &\quad+\int_{\partial B_R}\int_{\partial B_R} \mathbb E[T_{\rm M} (\boldsymbol E(x)\times\boldsymbol\nu(x))T_{\rm M} (\boldsymbol E(y)\times\boldsymbol\nu(y))^\top] \otimes( \boldsymbol U_1(x) \boldsymbol U_2(y)^{\top})\,\mathrm{d}s(x)\,\mathrm{d}s(y).
\end{align*}  Noting that \begin{align*}
    \mathbb F_1(x,y) &:=\mathbb E[(\boldsymbol{E}(x) \times \boldsymbol\nu(x))( \boldsymbol{E}(y) \times \boldsymbol\nu(y))^\top],\\ 
    \mathbb F_2(x,y) &:= \mathbb E[T_{\rm M} (\boldsymbol E(x)\times\boldsymbol\nu(x))( \boldsymbol{E}(y) \times \boldsymbol\nu(y))^\top], \\
    \mathbb F_3(x,y) &:=\mathbb E[T_{\rm M} (\boldsymbol E(x)\times\boldsymbol\nu(x))T_{\rm M} (\boldsymbol E(y)\times\boldsymbol\nu(y))^\top],
\end{align*} we derive the following estimate by an application of Cauchy--Schwartz inequality 
\begin{align*}
   &\left| \int_{\mathbb R^3} \sigma\boldsymbol U_1^{\top}\boldsymbol U_2{\rm d}x \right| \notag\\
   &\leq C \epsilon \Big(\|\boldsymbol U_1\|_{L^2(\partial B_R)^3}\|\boldsymbol U_2\|_{L^2(\partial B_R)^3}
+ \|\nabla \times \boldsymbol U_1\|_{L^2(\partial B_R)^3}\|\boldsymbol U_2\|_{L^2(\partial B_R)^3}\notag\\
&\quad+ \|\boldsymbol U_1\|_{L^2(\partial B_R)^3}\|\nabla \times \boldsymbol U_2\|_{L^2(\partial B_R)^3}
+\|\nabla \times \boldsymbol U_1\|_{L^2(\partial B_R)^3}\|\nabla \times \boldsymbol U_2\|_{L^2(\partial B_R)^3}\Big),
\end{align*} 
where $C$ is dependent on $k$ and $R$.
Then since the equations \eqref{Maxw} reduce to Helmholtz equations $(\Delta+k^2)\boldsymbol U_j=0$ near $\partial B_R$, by applying the trace theorem and 
standard interior regularity estimate of elliptic equations we obtain \eqref{key}. The proof is complete.
\end{proof}

Based on the estimate \eqref{key},
in the following we derive a stability estimate of recovering the strength $\sigma$ from the correlation data. 
To this end, we shall choose appropriate test functions $\boldsymbol U_1$ and $\boldsymbol U_2$  which are solutions to the associated Maxwell equation \eqref{Maxw}. 

The following lemma given by \cite{Colton1992}
provides us with useful complex geometric optics (CGO) solutions. 
\begin{lemma}\label{PHCGO}
   Let ${\boldsymbol \zeta},{\boldsymbol \eta} \in \mathbb C^3$ such that ${\boldsymbol \zeta} \cdot {\boldsymbol \zeta}=k^2$, ${\boldsymbol \zeta} \cdot {\boldsymbol \eta}=0$ and $|\Im {\boldsymbol \zeta}| \ge M_1$. Then there exists a solution to \[
   \nabla \times (\nabla \times\boldsymbol U)-k^2n \boldsymbol U=0 \quad\mathrm{in} \quad B_R
   \] satisfying $\boldsymbol U(x,{\boldsymbol \zeta},{\boldsymbol \eta})=e^{{\rm i}{\boldsymbol \zeta} \cdot x}({\boldsymbol \eta}+f(x,{\boldsymbol \zeta},{\boldsymbol \eta}){\boldsymbol \zeta}+\boldsymbol V(x,{\boldsymbol \zeta},{\boldsymbol \eta}))$, where the function and the vector field satisfy \begin{align}\label{cgoestimate}
       \|f(\cdot, {\boldsymbol \zeta},{\boldsymbol \eta})\|_{L^2(B_{R'})^3}+\|\boldsymbol V(\cdot, {\boldsymbol \zeta},{\boldsymbol \eta})\|_{L^2(B_{R'})^3} \le \frac{M_2|{\boldsymbol \eta}|}{|\Im {\boldsymbol \zeta}|}.
   \end{align} Here $M_1$ and $M_2$ are two constants depending on $k$ and $R'$ with $R'>R$.
\end{lemma}

Our main result on stability of the inverse random source problem is stated as follows.
\begin{theorem}\label{mainstab}
    Assume $\sigma$ satisfies the a priori regularity conditions
    \[
    \|\sigma\|_{L^\infty(\mathbb R^3)}, \, \|\sigma\|_{H^s(\mathbb R^3)} \le Q.
    \] with $s>0$ and $Q>0$ being two positive constants. Then there holds the stability estimate \[
     \|\sigma\|_{L^2(\mathbb R^3)} \le C \frac{1}{(-\log \epsilon)^\frac{4s}{7+2s}},
    \] where $C$ is a positive constant dependent on $k, Q, M_1, M_2$.
\end{theorem}
\begin{proof}
 We choose ${\boldsymbol \zeta}$ and ${\boldsymbol \eta}$ as follows.
For any given $\xi \in \mathbb R^3$, take ${\boldsymbol d}_1,{\boldsymbol d}_2 \in \mathbb S^2$ such that $\{\widehat \xi, {\boldsymbol d}_1,{\boldsymbol d}_2\}$ forms an orthonormal basis of $\mathbb R^3$. Under this basis, set \begin{align*}
    &{\boldsymbol \zeta}_1=\left(-\frac{|\xi|}{2},{\rm i}\sqrt{t^2-k^2+\frac{|\xi|^2}{4}},t\right),\quad {\boldsymbol \zeta}_2=\left(-\frac{|\xi|}{2},-{\rm i}\sqrt{t^2-k^2+\frac{|\xi|^2}{4}},-t\right), \\ &{\boldsymbol \eta}_1=\left(1,0,\frac{|\xi|}{2t}\right),\quad {\boldsymbol \eta}_2=\left(1,0,-\frac{|\xi|}{2t}\right),
\end{align*} where $t$ is a parameter satisfying $t \ge M_1+2k$. Then it follows from Lemma \ref{PHCGO} that there exist CGO solutions $\boldsymbol U_j=e^{{\rm i}{\boldsymbol \zeta}_j \cdot x}({\boldsymbol \eta}_j+f_j(x,{\boldsymbol \zeta}_j,{\boldsymbol \eta}_j){\boldsymbol \zeta}_j+\boldsymbol V_j(x,{\boldsymbol \zeta}_j,{\boldsymbol \eta}_j))$, $j=1,2$, satisfying \begin{align}
    \label{remainder}\|f(\dot,{\boldsymbol \zeta},{\boldsymbol \eta})\|_{L^2(B_{R'})}+\|\boldsymbol V(\dot,{\boldsymbol \zeta},{\boldsymbol \eta})\|_{L^2(B_{R'})^3} \le \frac{M_2|{\boldsymbol \eta}_j|}{|\Im {\boldsymbol \zeta}_j|} \le   \frac{2M_2+\frac{M_2|\xi|}{M_1}}{t}.
\end{align}
A direct calculation gives \begin{align*}
    \|\boldsymbol U_j\|_{L^2( B_{R'})^3} &\le \|e^{{\rm i}{\boldsymbol \zeta}_j \cdot x} {\boldsymbol \eta}_j\|_{L^2( B_{R'})^3}+\|e^{{\rm i}{\boldsymbol \zeta}_j \cdot x}f {\boldsymbol \zeta}_j\|_{L^2( B_{R'})^3}+\|e^{{\rm i}{\boldsymbol \zeta}_j \cdot x}\boldsymbol V\|_{L^2( B_{R'})^3} \\ &\le C e^{(t+|\xi|)R'}\left(1+\frac{|\xi|}{t} \right).
\end{align*}
Inserting the CGO solutions constructed above into \eqref{key} yields \begin{align}\label{left}
    \left|\int_{\mathbb R^3 } \sigma \boldsymbol U_1 \cdot \boldsymbol U_2\,\mathrm{d}x\right| \le C e^{2(t+|\xi|)R'}\left(1+\frac{|\xi|}{t}\right)\epsilon.
\end{align}
Notice that \begin{align*}
    \boldsymbol U_1(x,{\boldsymbol \zeta},{\boldsymbol \eta}) \cdot \boldsymbol U_2(x,{\boldsymbol \zeta},{\boldsymbol \eta})=e^{-{\rm i}\xi \cdot x}\Big(1-\frac{|\xi|^2}{4t^2}+r(x)\Big),
\end{align*} where $r(x)$ is given by \begin{align*}
    r(x) &=-|\xi|(f(x,{\boldsymbol \zeta}_1,{\boldsymbol \eta}_1)+f(x,{\boldsymbol \zeta}_2,{\boldsymbol \eta}_2))+f(x,{\boldsymbol \zeta}_1,{\boldsymbol \eta}_1)f(x,{\boldsymbol \zeta}_2,{\boldsymbol \eta}_2)(|\xi|^2/2-k^2) \\ &\quad+\boldsymbol V(x,{\boldsymbol \zeta}_1,{\boldsymbol \eta}_1) \cdot \boldsymbol V(x,{\boldsymbol \zeta}_2,{\boldsymbol \eta}_2) +{\boldsymbol \eta}_1 \cdot \boldsymbol V(x,{\boldsymbol \zeta}_2,{\boldsymbol \eta}_2) + {\boldsymbol \eta}_2 \cdot \boldsymbol V(x,{\boldsymbol \zeta}_1,{\boldsymbol \eta}_1)\\ &\quad+f(x,{\boldsymbol \zeta}_2,{\boldsymbol \eta}_2){\boldsymbol \eta}_2 \cdot\boldsymbol V(x,{\boldsymbol \zeta}_1,{\boldsymbol \eta}_1)+f(x,{\boldsymbol \zeta}_1,{\boldsymbol \eta}_1){\boldsymbol \eta}_1 \cdot\boldsymbol V(x,{\boldsymbol \zeta}_2,{\boldsymbol \eta}_2).
\end{align*}
By  applying \eqref{remainder} we get  \begin{align}\label{right}
    \|r\|_{L^1(B_{R'})} \le C \left(\frac{|\xi|^2+|\xi|}{t}+\frac{|\xi|^4+1}{t^2}\right) \le C \frac{1+|\xi|^4}{t^2}
\end{align} since $t \ge M_1+2k$.
Since $\|\sigma\|_{L^\infty(\mathbb R^3)} \le Q$, combining \eqref{left} and \eqref{right} yields \begin{align*}
    |\widehat\sigma(\xi)|=\left|\int_{\mathbb R^3 } \sigma(x)  e^{-{\rm i}\xi\cdot x}\,\mathrm{d}x\right| \le C \left(e^{2(t+|\xi|)R'}\left(1+\frac{|\xi|}{t}\right)\epsilon +\frac{1+|\xi|^4}{t^2} \right).
\end{align*}  
Then noting $\|\sigma\|_{H^s(\mathbb R^3)} \le Q$ we arrive at \begin{align*}
    \|\sigma\|_{L^2(\mathbb R^3)} &\le C \left(\left|\int_{|\xi|<\rho} |\widehat \sigma(\xi)|^2\,\mathrm{d}\xi\right| + \left|\int_{|\xi| \ge \rho} |\widehat \sigma(\xi)|^2\,\mathrm{d}\xi\right| \right) \\ &\le C \left( \rho^3\left(e^{2(t+\rho)R'}\left(1+\frac{\rho}{t}\right)\epsilon +\frac{1+\rho^4}{t^2} \right)+ \frac{\|\sigma\|_{H^s(\mathbb R^3)}}{\rho^{2s}} \right) \\ &\le C\left( \rho^3e^{2(t+\rho)R'}\left(1+\frac{\rho}{t}\right)\epsilon +\frac{\rho^7}{t^2} + \frac{1}{\rho^{2s}} \right).
\end{align*} Here $\rho>1$ is a positive constant. By letting $\rho=t^{\frac{2}{7+2s}}$ we obtain \begin{align*}
    \|\sigma\|_{L^2(\mathbb R^3)} &\le C \left( t^\frac{6}{7+2s}e^{2(t+t^\frac{2}{7+2s})R'}\left(1+t^{-\frac{5+2s}{7+2s}}\right)\epsilon +t^{-\frac{4s}{7+2s}}  \right) \le C \left( t^\frac{6}{7+2s}e^{4tR'}\epsilon +t^{-\frac{4s}{7+2s}}  \right)
\end{align*} for $t>1$. Without loss of generality, by assuming $\epsilon$ is a sufficiently small constant, we are able to choose $t=-\frac{1}{2R'}\log \epsilon$ which yields $t>\{1,M_1+2k\}$ and $\rho=t^{\frac{2}{7+2s}}>1$. As a result, we derive the stability estimate \begin{align*}
     \|\sigma\|_{L^2(\mathbb R^3)} \le C \frac{1}{(-\log \epsilon)^\frac{4s}{7+2s}},
\end{align*} 
and it can be seen from the proof that the constant $C$ depends on $k, Q, M_1, M_2$. The proof is complete.
\end{proof}
\begin{remark}
The stability implies the uniqueness of the inverse random source problems for the Maxwell equation at a fixed frequency. It is of logarithmic type which illustrates that the inverse problem is ill-posed. Moreover, the stability increases if the regularity of the strength improves.
\end{remark}

\section{Conclusion}

In this paper, we have investigated the inverse random source problem for Maxwell equations in an inhomogeneous medium. 
We have established the well-posedness of the direct problem and obtained the estimates and regularity of the solution.
By using the complex geometric optics solutions to the associated Maxwell equation, we have obtained a logarithmic stability of determining the strength
of the random source by using Dirichlet data only at a single frequency. The methods developed this work can be used to study other stochastic wave equations
in inhomogeneous media or with potential functions such as the elastic wave equation and Schr\"odinger equations.
A more challenging direction is to study the case where the strength functions
are different for each component of the random source. In this case, the constructions of the complex geometric optics solutions are much more complicated.
We hope to be able to report progresses on these problems in future works.


\begin{thebibliography}{99}

\bibitem{Monk}
R. Albanese and P. Monk, The inverse source problem for Maxwell's equations, Inverse Problems, 22 (2006), 1023--1035.

\bibitem{abf}
H. Ammari, G. Bao, J. Fleming, An inverse source problem for Maxwell's equations in magnetoencephalography, SIAM J. Appl. Math., 62 (2002) 1369--1382.

\bibitem{acl}
F. Assous, P. Ciarlet, and S. Labrunie,
Mathematical foundations of computational electromagnetism, 
Applied Mathematical Sciences (AMS, volume 198), 2018 Springer.

  


\bibitem{blz}
G. Bao, P. Li and Y. Zhao, 
Stability for the inverse source problems in elastic
and electromagnetic waves, 
 J. Math. Pures Appl., 134 (2020), 122--178.

\bibitem{CK}
D. Colton and R. Kress, Integral Equation Methods in Scattering Theory, Pure and Applied
Mathematics (New York), A Wiley-Interscience Publication, John Wiley and Sons, Inc., New
York, 1983.




\bibitem{LZZ}
P. Li, J. Zhai and Y. Zhao, 
\newblock Stability for the acoustic inverse source problem in inhomogeneous media, 
\newblock {\em SIAM J. Appl. Math.}, 80 (2020), 2547--2559.


\bibitem{Caro2010}
P. Caro,
Stable determination of the electromagnetic coefficients by boundary measurements.
Inverse Problems, 26 (2010), 105014. 

\bibitem{De}
A. Devaney, The inverse problem for random sources, J. Math. Phys., 20 (1979), 1687--1691.


 \bibitem{Colton1992} 
 D. Colton and L. P\"aiv\"arinta
The uniqueness of a solution to an inverse scattering problem for electromagnetic waves.
Arch. Rational Mech. Anal., 119 (1992), 59--70.

\bibitem{DZ}
S. Dyatlov and M. Zworski, Mathematical Theory of Scattering Resonances, vol. 200, American Mathematical Soc., 2019.

\bibitem{isakov}
V. Isakov, Inverse Source Problems, AMS, Providence, RI, 1990.


\bibitem{LL}
P. Li and Y. Liang,  Stability for inverse source problems of the stochastic Helmholtz equation with a white noise, SIAM
J. Appl. Math., 84 (2024), 687--709.

\bibitem{LLW}
J. Li, P. Li, and X. Wang, Inverse source problems for the stochastic wave equations: far-field patterns, SIAM
J. Appl. Math., 82 (2022),  1113--1134.

\bibitem{LW_Maxwell}
P. Li and X. Wang, An inverse random source problem for Maxwell's equations, Multiscale Model. Simul., 19 (2021), 25--45.

\bibitem{LZZ}
P. Li, J. Zhai, and Y. Zhao, Stability for the acoustic inverse source problem in inhomogeneous media, SIAM J. Appl. Math., 80 (2020), 2547--2559.

\bibitem{MKB}
E. Marengo, M. Khodja, and A. Boucherif,
Inverse source problem in nonhomogeneous background media. Part II: vector formulation and antenna substrate performance characterization,
69(2008), SIAM J. Appl. Math., 81--110.

\bibitem{Nara}
T. Nara, J. Oohama, M. Hashimoto, T. Takeda, S. Ando, Direct reconstruction algorithm of current dipoles for vector magnetoencephalography and electroencephalography, Phys. Med. Biol., 52 (2007) 3859--3879.

\bibitem{Stein}
E. Stein, Harmonic Analysis: Real-Variable Methods, Orthogonality, and Oscillatory Integrals,
Princeton University Press, Princeton, NJ, 1993.

\bibitem{Zhao}
Y. Zhao, Stability for the electromagnetic inverse source problem in inhomogeneous media, J. Inverse Ill-Posed Probl., 31 (2023), 103--116.


\end{thebibliography}
\end{document}